\documentclass[11pt, letterpaper, reqno]{amsart}

\usepackage[a4paper,
            top=30mm,bottom=30mm,
            left=30mm,right=30mm]{geometry}

\usepackage{amsmath, amsthm, amsfonts, amssymb, amscd, mathtools, bm} 
\usepackage{xcolor} 
\usepackage[unicode=true]{hyperref} 
\usepackage[shortlabels]{enumitem} 
\usepackage[T1]{fontenc}
\usepackage{empheq}
\usepackage{placeins}

\usepackage{tikz} 
\usepackage{graphicx}  
\usepackage{subcaption, wrapfig}

\theoremstyle{plain}
\newtheorem{theorem}{Theorem}[section]
\newtheorem{lemma}[theorem]{Lemma}

\theoremstyle{definition}

\theoremstyle{remark}

\numberwithin{equation}{section}

\newcommand{\rd}{\mathrm{d}}

\newcommand{\R}{\mathbb{R}}
\newcommand{\Z}{\mathbb{Z}}

\title[]{Uniform boundedness of parametric \\ bilinear fractional integrals}

\author[]{Nuno J. Alves}
\author[]{Loukas Grafakos}

\address[N. J. Alves]{
      University of Vienna, Faculty of Mathematics, Oskar-Morgenstern-Platz 1, 1090 Vienna, Austria.}
\email{nuno.januario.alves@univie.ac.at}

\address[L. Grafakos]{
      University of Missouri, Department of Mathematics, Columbia MO 65211, USA.}
\email{grafakosl@umsystem.edu}

\begin{document}

\begin{abstract}
We provide weak-type bounds for a family of bilinear fractional integrals that 
arise in the study of Euler--Riesz systems. These bounds are uniform in the natural parameter that describes the family and are sharp, in the sense that they do not hold for any larger set of indices.
\end{abstract}

\keywords{bilinear fractional integrals, endpoint estimates, interpolation}
\subjclass[2020]{42B20, 46E30} 
\maketitle
\thispagestyle{empty} 

\section{Introduction}

Let $d \in \mathbb{N}$ and \( 0 < \alpha < d \). For \( \theta \in [0,1] \), we consider a bilinear fractional integral operator \( I_\alpha^\theta \), defined for nonnegative measurable functions \( f \) and \( g \) on \( \mathbb{R}^d \) by
\begin{equation} \label{operatorI}
I_\alpha^\theta(f,g)(x) = \int_{\R^d} f(x + (\theta - 1)y) \, g(x + \theta y) \, |y|^{\alpha - d} \, \rd y .
\end{equation} \par 
This operator was introduced in \cite{agt} by the authors and Tzavaras, motivated by a reformulation of the Euler--Riesz system that leads to an a priori gain of integrability for the density. The relevant equations model the evolution of a compressible fluid with a nonlocal repulsive potential of Riesz type and read as
\begin{equation} \label{ER}
\begin{cases}
\partial_t \rho + \nabla \cdot (\rho u) = 0, \\
\partial_t (\rho u) + \nabla \cdot (\rho u \otimes u) + \nabla \rho^\gamma + \rho \nabla K_\alpha \ast \rho = 0,
\end{cases}
\end{equation}
where $\rho \colon [0,\infty) \times \mathbb{R}^d \to [0,\infty)$ denotes the density, $u \colon [0,\infty) \times \mathbb{R}^d \to \mathbb{R}^d$ the linear velocity, and $\gamma > 1$ the adiabatic exponent. We use the notation $u \otimes u$ for the matrix with entries $u_i u_j$. The kernel $K_\alpha$ is given by
\begin{equation*}
K_\alpha(x) = \tfrac{1}{d-\alpha} |x|^{\alpha - d} 
\end{equation*}
and describes the nonlocal interaction of particles. \par 
A smooth solution $(\rho,u)$ of \eqref{ER} with fast decay at infinity satisfies the conservation of energy and mass identities:
\begin{equation}\label{apbounds}
\frac{\mathrm{d}}{\mathrm{d}t} \int_{\mathbb{R}^d} \tfrac{1}{2}\rho |u|^2 + \tfrac{1}{\gamma - 1} \rho^\gamma + \tfrac{1}{2} \rho (K_\alpha \ast \rho) \, \mathrm{d}x = 0, \qquad \frac{\mathrm{d}}{\mathrm{d}t} \int_{\mathbb{R}^d} \rho \, \mathrm{d}x = 0 ,
\end{equation}
which provides an a priori estimate for solutions. This, in particular, implies the following regularity for the density: 
\begin{equation}  \label{apriori_reg}
\rho \in L^\infty\big((0,\infty);L^1 \cap L^\gamma(\mathbb{R}^d)\big).
\end{equation} 
\par 
Interestingly, by exploiting the structure of the equations, and in particular the symmetry of the interaction kernel $K_\alpha$, it is possible to improve this a priori regularity of the density for finite-energy solutions. The key idea is to rewrite system~\eqref{ER} as a space-time divergence-free condition for a suitable positive-definite symmetric tensor. Then, a gain in integrability for the density $\rho$ follows from Serre’s theory of \textit{compensated integrability}~\cite{serre2018divergence,serre2019compensated}, as established in~\cite[Theorem~3.2]{agt}, yielding
\begin{equation} \label{improved_reg}
\rho \in L^{\gamma + \frac1d}\big((0,\infty) \times \mathbb{R}^d\big).
\end{equation}
\par 
 Among all terms in the system, the only one that is not in divergence form is the nonlinear nonlocal interaction $\rho \nabla K_\alpha \ast \rho$. Using the symmetry of $K_\alpha$, this term can be formally rewritten as the divergence of a positive-definite symmetric tensor:
\begin{equation} \label{Sind}
\rho \nabla K_\alpha \ast \rho = \nabla \cdot S_\alpha(\rho)  
\end{equation}
where 
\begin{equation*} \label{S}
S_\alpha(\rho) = \tfrac{1}{2} \int_0^1 \int_{\R^d} \rho(x + (\theta - 1)y) \, \rho(x + \theta y) \, |y|^{\alpha - d - 2} \, y \otimes y \, \rd y \, \rd \theta .
\end{equation*}
By a tensor we mean a matrix-valued function. See \cite[Appendix A]{agt} for the details of this derivation. \par 
The structural identity \eqref{Sind} allows the entire Euler--Riesz system to be cast in divergence form,
\begin{equation} \label{tensorER}
\nabla_{t,x} \cdot
\begin{bmatrix}
\rho & (\rho u)^\top  \\
\rho u & \rho u\otimes u + \rho^\gamma I_d + S_\alpha(\rho)  \\
\end{bmatrix} = 0,
\end{equation}
where $I_d$ is the $d\times d$ identity matrix. The tensor in~\eqref{tensorER} is a divergence-free positive symmetric tensor, thus fitting into the framework of compensated integrability.
 \par 
  The operator $I_\alpha^\theta$ thus arises naturally in the study of the integrability mapping properties of $S_\alpha$, a key step in understanding the nonlocal term of the reformulated Euler--Riesz system~\eqref{tensorER}. The main result of \cite{agt} establishes that $I_\alpha^\theta$ maps $L^p(\R^d) \times L^q(\R^d)$ to $L^r(\R^d)$ uniformly in $\theta$ under some conditions on $p,q,r$. The result is as follows:

\begin{theorem}[{\cite{agt}}] \label{thm_uniform_strong_bounds} 
Let \( p, q, r \) be integrability exponents satisfying $1 < p, q < d/\alpha$ and
\begin{equation} \label{r_exponent}
\frac{1}{p}+ \frac{1}{q} = \frac{1}{r} + \frac{\alpha}{d} .
\end{equation}
Then there is a constant $C = C(\alpha, d, p, q) > 0$ independent of $\theta$ such that 
\begin{equation}\label{eq_uniform_strong_bounds}
\big\| I_\alpha^\theta(f, g) \big\|_r \leq C \, \| f \|_p \, \| g \|_q 
\end{equation}
for all $f \in L^p(\mathbb{R}^d)$ and $g \in L^q(\mathbb{R}^d)$.
\end{theorem}

Theorem \ref{thm_uniform_strong_bounds} provides uniform strong-type boundedness for $I_\alpha^\theta$ when the pair $(1/p, 1/q)$ lies inside the square given by the convex hull of $\{(1,1),(1,\alpha/d),(\alpha/d,1),(\alpha/d,\alpha/d)\}$. In this work, we complement this result in two directions: First, we prove uniform weak-type endpoint estimates along the boundary of this square; interestingly, different edges of the square exhibit distinct behaviors with respect to the uniformity of the estimates in the parameter $\theta$; see Theorems \ref{thm_weak_bounds} and \ref{thm_lorentz}. Secondly,  we establish strong-type boundedness over a larger range of exponents, corresponding to the interior of a pentagon --- these estimates are uniform in $\theta$ away from $0$ and $1$; see Theorem \ref{thm_strong_bounds}. This completes the mapping properties of $I_\alpha^\theta$ on Lebesgue and Lorentz spaces initiated in \cite{agt}.
\par 
Bilinear fractional integral operators have been an important object of study in harmonic analysis over the past three decades. Their significance comes from their singular nature and the challenge of determining whether such operators admit bounds in terms of the norms of the functions on which they act. This question is far from trivial and over the years, a variety of boundedness results have been obtained.
\par
One of the earliest and most studied examples is the operator $B_\alpha$, defined for nonnegative measurable functions by
\begin{equation} \label{operatorB}
B_\alpha(f,g)(x) = \int_{\mathbb{R}^d} f(x - y) \, g(x + y) \, |y|^{\alpha - d} \, \mathrm{d}y  .
\end{equation}
This operator was introduced by the second author in~\cite{grafakos1992multilinear} and further investigated by Kenig and Stein~\cite{kenig1999multilinear}, as well as by the second author and Kalton~\cite{grafakos2001some}, where bounds in Lebesgue spaces were established. Extensions of these boundedness results for $B_\alpha$ have been obtained in the rough kernel case~\cite{ding2002rough}, on weighted Lebesgue spaces~\cite{moen2014new, li2016two, hoang2018weighted, furuya2020weighted}, and on Morrey spaces~\cite{hatano2019note,he2021bilinear}.
\par 
We note that $I_\alpha^\theta(f,g)$ interpolates between $I_\alpha(f) \, g$ and $f \, I_\alpha(g)$, which correspond to the cases $\theta = 0$ and $\theta = 1$, respectively, where $I_\alpha$ is the classical singular fractional integral (or Riesz potential)
\begin{equation} \label{operatorSI}
I_\alpha(f)(x) = \int_{\R^d} f(x - y) \, |y|^{\alpha - d} \, \rd y .
\end{equation}
The mid point of this interpolation corresponds to the operator $B_\alpha$ through the identity
\begin{equation} \label{relBI}
I_\alpha^{1/2}(f,g) = 2^\alpha B_\alpha(f,g)  .
\end{equation}
Thus, the known bounds for the operator $B_\alpha$ are easily recovered from the results we obtain for $I_\alpha^\theta$. We emphasize, however, that the search for uniform bounds of the $I_\alpha^\theta$ makes its study more intricate than that of $B_\alpha$. Moreover, extensions of these uniform estimates to weighted $L^p$ or Morrey spaces would provide a broader context for this family of operators and may be of independent interest. 
\par 
The manuscript is organized as follows. Section~\ref{section_main_results} contains the statement of the main results. In Section~\ref{section_prelim}, we consider an auxiliary operator $I_j^\theta$ and deduce several boundedness results for it. The proofs of Theorems~\ref{thm_weak_bounds} and~\ref{thm_strong_bounds} are given in Section~\ref{section_proofs}, and the proof of Theorem~\ref{thm_lorentz} appears in Section~\ref{section_proof_lorentz}. In the final part of this work, Section~\ref{section_sharp}, we investigate the sharpness of Theorem~\ref{thm_weak_bounds}.

\vspace{2mm}
\noindent \textbf{Notation.} \par
We explain the notation used in the statements of the theorems and throughout the text.
\par For $0 < p \leq \infty$, we denote by $L^p(\R^d)$ the standard Lebesgue space of measurable functions on $\R^d$, equipped with the norm $\|\cdot\|_p$. The weak Lebesgue space is denoted by $L^{p,\infty}(\R^d)$, with quasi-norm $\|\cdot\|_{p,\infty}$, and $L^{p,1}(\R^d)$ stands for the Lorentz space with indices $p$ and $1$, whose quasi-norm is denoted by $\|\cdot\|_{p,1}$. For a thorough exposition of the basic properties of these spaces, we refer the reader to \cite[Chapter 1]{grafakos2014classical}.

\par
Given $R > 0$ and $x \in \R^d$, the open ball centered at $x$ with radius $R$ is denoted by $B_x(R)$; for $x = 0$, we simply write $B(R)$.

\par
If $E \subseteq \R^d$ is measurable, its Lebesgue measure is denoted by $|E|$, and its characteristic function by $\chi_E$.

\par
The symbol $c$ denotes (possibly different) positive constants depending at most on $\alpha$, $d$, or both. Similarly, $C$ denotes positive constants that may depend on $\alpha$ and $d$, as well as on the integrability exponents $p,q$.

\section{Main results} \label{section_main_results}

Our first main result, stated below, provides the complete range of exponents for which weak-type bounds hold uniformly in $\theta$.

 \begin{theorem} \label{thm_weak_bounds}
Let $0 < \delta \leq 1/2$. The following assertions hold:
\begin{enumerate}[(i)]
\item If $1\le p  < d/\alpha  $, then there is a constant $C = C( \alpha, d, p) > 0$ independent of $\theta$ such that 
\begin{equation} \label{ceiling}
\big\|I_\alpha^\theta(f,g) \big\|_{r, \infty} \leq C \, \|f \|_p \, \|g \|_1
\end{equation}
for all $f \in L^p(\R^d)$ and $g \in L^1(\R^d)$, where $1/r=1/p+1-\alpha/d$.
\item If $1 \leq q   < d/\alpha$, then there is a constant $C = C( \alpha, d, q)>0$ independent of $\theta$ such that 
\begin{equation} \label{outer_wall}
\big\|I_\alpha^\theta(f,g) \big\|_{r, \infty} \leq C \, \|f \|_1 \, \|g \|_q
\end{equation}
for all $f \in L^1(\R^d)$ and $g \in L^q(\R^d)$, where $1/r=1/q+1-\alpha/d$.
\item If $0 \leq  \theta \leq 1 - \delta$, and $1 \leq p <   d/\alpha$, then there is a constant $C = C(\delta, \alpha, d, p) > 0$ such that 
\begin{equation} \label{floor}
\big\|I_\alpha^\theta(f,g) \big\|_{p, \infty} \leq C \, \|f \|_p \, \|g \|_\frac{d}{\alpha}
\end{equation}
for all $f \in L^p(\R^d)$ and $g \in L^{\frac{d}{\alpha}}(\R^d)$.
\item If $\delta \leq \theta \leq 1 $ and $1 \leq q <   d/\alpha$, then there is a constant $C = C(\delta, \alpha, d, q)>0$ such that 
\begin{equation} \label{inner_wall}
\big\|I_\alpha^\theta(f,g) \big\|_{q, \infty} \leq C \, \|f \|_{\frac{d}{\alpha}} \, \|g \|_q
\end{equation}
for all $f \in L^{\frac{d}{\alpha}}(\R^d)$ and $g \in L^q(\R^d)$.
\item If $\delta \leq \theta \le 1-\delta$, then there is a constant $C = C(\delta, \alpha, d)>0$ such that 
\begin{equation} \label{bad_corner}
\big\| I_\alpha^\theta(f,g) \big\|_{\frac{d}{\alpha}, \infty } \leq C \, \|f \|_{\frac{d}{\alpha}} \, \|g \|_{\frac{d}{\alpha}}
\end{equation}
for all $f \in L^{\frac{d}{\alpha}}(\R^d)$ and $g \in L^{\frac{d}{\alpha} }(\R^d)$.
\end{enumerate}
\end{theorem}

\par 
In Theorem \ref{thm_weak_bounds},  the different edges of the square inside which uniform strong-type bounds hold exhibit varying behavior with respect to the uniformity in $\theta$ of the estimates; see Figure 1.  The weak-type estimates are uniform in $\theta \in [0,1]$ on the upper and right edges, uniform in $\theta$ away from $1$ on the lower edge, and uniform in $\theta$ away from $0$ on the left edge. The southwestern corner is where the behavior is most delicate as uniformity holds {\it only} away from both critical extremes of the unit interval. \par

At this point, one may wonder whether full uniform weak-type estimates can hold on the ``bad'' edges and corners at the cost of restricting the domain of the operator $I_\alpha^\theta$. The answer is yes, as long as the domains are  appropriate Lorentz spaces $L^{p,1}$. This as shown in the next theorem.

\begin{theorem} \label{thm_lorentz}
The following estimates hold uniformly in $\theta$:
\begin{enumerate}[(i)]
\item If $1 \leq p \leq d/\alpha$, then there is a constant $C = C(\alpha,d,p) > 0$ such that 
\begin{equation}\label{floor_uniform}
\big\| I_\alpha^\theta(f,g) \big\|_{p,\infty} \leq C \, \|f \|_{p,1} \, \|g \|_{\frac{d}{\alpha},1}
\end{equation} 
for all $f \in L^{p,1}(\R^d)$ and $g \in L^{\frac{d}{\alpha},1}(\R^d)$.
\item If $1 \leq q \leq d/\alpha$, then there is a constant $C = C(\alpha,d,q) > 0$ such that 
\begin{equation}\label{inner_wall_uniform}
\big\| I_\alpha^\theta(f,g) \big\|_{q,\infty} \leq C \, \|f \|_{\frac{d}{\alpha},1} \, \|g \|_{q,1}
\end{equation} 
for all $f \in L^{p,1}(\R^d)$ and $g \in L^{\frac{d}{\alpha},1}(\R^d)$.
\end{enumerate}
\end{theorem}

Our last result provides strong-type bounds in the best possible range of exponents (a pentagon), uniformly away from the ``bad directions''. This is expected, since uniform bounds do not hold in this pentagon. 

\begin{theorem} \label{thm_strong_bounds}
Let $1\le p, q \le \infty$ be integrability exponents such that $(1/p,1/q)$ lies in the pentagon determined by the interior of the convex hull of $\{(1,0),(0,1),(\alpha/d,0), (0,\alpha/d), (1,1)\}$, and let $r$ be as in \eqref{r_exponent}. If $0 < \delta \leq 1/2$ and $\delta \leq  \theta \leq 1 - \delta$, then there is a constant $C = C(\delta, \alpha, d, p ,q) > 0$ such that 
\begin{equation}\label{eq_strong_bounds}
\big\| I_\alpha^\theta(f, g) \big\|_r \leq C \, \| f \|_p \, \| g \|_q 
\end{equation}
for all $f \in L^p(\mathbb{R}^d)$ and $g \in L^q(\mathbb{R}^d)$.
\end{theorem}
\medskip

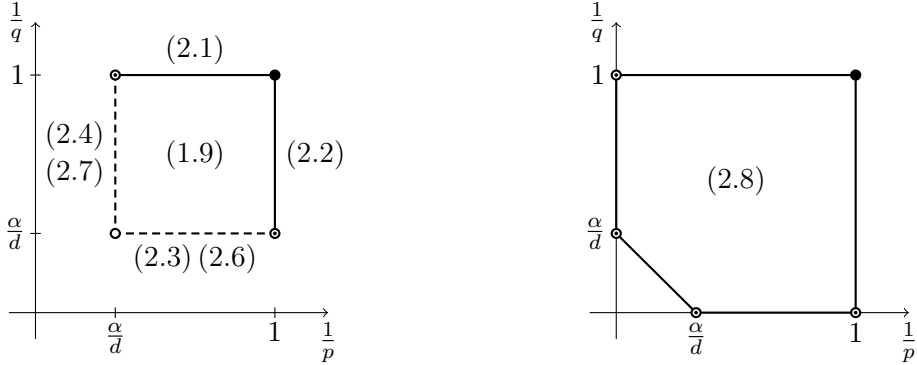
\begin{figure}[h]
  \centering
\tikzset{point/.style={fill=white, draw, circle, inner sep=1.2pt}, pointsm/.style={circle, fill=black, inner sep=0.5pt}}
\begin{subfigure}{0.48\textwidth}
    \centering
    \begin{tikzpicture}[scale=.35]
     \draw[->] (0,-1) -- (0,11) node[left] {$\frac 1 q$}; 
     \draw[->] (-1,0) -- (11,0) node[below] {$\frac 1 p$}; 

\draw[dash=on 3pt off 1.9pt phase 0.9pt, thick] 
  (3,9) -- node[midway, left]{\shortstack{\eqref{inner_wall} \\ \eqref{inner_wall_uniform}}} 
  (3,3);

\draw[dash=on 3pt off 1.9pt phase 0.9pt, thick] (3,3) -- node[midway, below]{\eqref{floor}\,\eqref{floor_uniform}} (9,3); 

\draw[thick] (9,3) node[point]{} node[pointsm]{} -- node[midway, right]{\eqref{outer_wall}} 
             (9,9) node[point,fill=black]{} -- node[midway, above]{\eqref{ceiling}} 
             (3,9) node[point]{} node[pointsm]{} 
             (3,3) node[point]{};
	
\draw (9,-0.2)--(9,0.2) (3,-0.2)--(3,0.2) (-0.2,3)--(0.2,3) (-0.2,9)--(0.2,9); 

\node at (9,0) [below] {$1$};
\node at (3,0) [below] {$\frac{\alpha}{d}$};
\node at (0,3) [left] {$\frac{\alpha}{d}$};
\node at (0,9) [left] {$1$};
\node at (6,6) [] {$\eqref{eq_uniform_strong_bounds}$};

    \end{tikzpicture}
  \end{subfigure} 
\ \
\begin{subfigure}{0.48\textwidth}
    \centering
    \begin{tikzpicture}[scale=.35]

\draw[->] (0,-1) -- (0,11) node[left] {$\frac 1 q$}; 
\draw[->] (-1,0) -- (11,0) node[below] {$\frac 1 p$}; 

\draw[thick] (3,0)  -- 
	         (9,0) node[point]{} node[pointsm]{} node[below]{$1$} -- 
	         (9,9) node[point,fill=black]{} -- 
	         (0,9) node[point]{} node[pointsm]{} node[left]{$1$} -- 
	         (0,3); 

\draw[thick] (0,3) node[point]{} node[pointsm]{} node[left]{$\frac \alpha d$} -- 
	         (3,0) node[point]{} node[pointsm]{} node[below]{$\frac \alpha d$}; 

\node at (4.5,5) [] {$\eqref{eq_strong_bounds}$};
   \end{tikzpicture}
  \end{subfigure}  
\caption{Each numbered estimate above corresponds to boundedness on an open region or a boundary segment. The bounds in Theorems~\ref{thm_weak_bounds} and
 ~\ref{thm_lorentz} are depicted in the square on the left, while the region of boundedness in Theorem~\ref{thm_strong_bounds} is shown in the figure on the right.}
\end{figure}
\FloatBarrier

\section{Preliminary estimates} \label{section_prelim}
For $j \in \Z$, let $I_j^\theta$ be the operator defined for nonnegative measurable functions $f$ and $g$ on $\R^d$ by
\begin{equation} \label{Ithetaj}
I_j^\theta(f,g)(x) = \int_{|y|\leq 2^j} f(x + (\theta -1)y) \, g(x + \theta y) \, \mathrm{d}y \, .
\end{equation}

The first lemma, proved in \cite{agt}, yields that the operator $I_j^\theta$ maps $L^1(\R^d) \times L^1(\R^d)$ to both $L^1(\R^d)$ and $L^{\frac{1}{2}}(\R^d)$ uniformly in the parameter $\theta$.
\begin{lemma}
The operator $I_j^\theta$ satisfies the following estimates uniformly in $\theta$:
\begin{align}
\big\|I_j^\theta(f, g) \big\|_{1} &\leq \|f\|_1 \, \|g\|_1 \label{aux0} \, ,
\\
\big\|I_j^\theta(f, g) \big\|_{\frac{1}{2}} & \leq c \, 2^{dj} \, \|f\|_1 \, \|g\|_1 \, .\label{aux1}
\end{align}
\end{lemma}
\begin{proof}
See \cite[Lemma 4.2 and Lemma 4.3]{agt}.
\end{proof}

\begin{lemma} \label{lem_00K}
For any $1\le p\le \infty$, 
the operator $I_j^\theta$ satisfies the following estimates uniformly in $\theta$:
\begin{empheq}[left={\displaystyle \big\|I_j^\theta(f, g) \big\|_{\frac{p}{p+1}} \leq \, \empheqlbrace}]{align} 
    & c \, 2^{dj} \, \|f\|_1 \, \|g\|_p \, , \label{aux20}
\\ 
    & c \, 2^{dj} \, \|f\|_p \, \|g\|_1 \, . \label{aux21}
  \end{empheq}
\end{lemma}

\begin{proof}
Estimates \eqref{aux20} and \eqref{aux21} are symmetrical; we provide a proof of the former. Let $f \in L^1(\R^d)$ and $g \in L^\infty(\R^d)$ be nonnegative. By Fubini's theorem and the change of variables $z = x + (\theta - 1) y$, we have:
\begin{align*}
\big\|I_j^\theta(f,g) \big\|_1 & \leq \|g \|_\infty \int_{|y| \leq 2^j} \int_{\R^d} f(x + (\theta - 1) y) \, \rd x \, \rd y \\
& \leq \|g \|_\infty \int_{|y| \leq 2^j} \int_{\R^d} f(z) \, \rd z \, \rd y \\
& = c \, 2^{dj} \|g \|_\infty \|f \|_1.
\end{align*}
Estimate \eqref{aux20} follows from this bound together with \eqref{aux1}, by interpolation; see Theorem~\ref{Cinterpolation} in the appendix. 
\end{proof}

\begin{lemma}\label{OOKK}
Let $E$ be a measurable subset of $\R^d$ with $|E| < \infty$. The following estimates are valid for $1\le p\le \infty$: 
\begin{empheq}[left={\displaystyle 
\big\|I_j^\theta(f,g) \chi_E \big\|_{\frac{1}{2}} \leq \, \empheqlbrace}]{align} 
    & C\, (2^{dj}|E|)^{1-\frac1{p}} \, \| f\|_p \, \|g\|_1 \, \min \{2^{dj} , |E|\}^{\frac1p}  \, , \label{L0} \\
   & C\, (2^{dj}|E|)^{1-\frac1{p}} \, \| f\|_1 \, \|g\|_p \, \min \{ 2^{dj} , |E| \}^{\frac1p}  \, , \label{L00}  \\
   & C\, 2^{dj\left(1- \frac{1}{p} \right)} \, |E|^{2-\frac{\alpha}{d}-\frac1{p}} \, \| f\|_p \, \|g\|_{\frac{d}{\alpha}} \, \min \left\{2^{dj} , \frac{|E|}{(1-\theta)^{d-\alpha}} \right\}^{\frac1p}  \, , \label{L000} \\
  & C\, 2^{dj\left(1- \frac{1}{p} \right)} \, |E|^{2-\frac{\alpha}{d}-\frac1{p}} \, \| f\|_{\frac{d}{\alpha}} \, \|g\|_p \, \min \left\{2^{dj} , \frac{|E|}{\theta^{d-\alpha}} \right\}^{\frac1p}  \, , \label{L0000}
  \end{empheq}
for some $C = C(\alpha,d,p)>0$ independent of $\theta$.
\end{lemma}

\begin{proof}
We start by deducing \eqref{L0}. Observe that the change of variables $z=x+\theta y$ gives
\[
\int_{\mathbb R^d} I_j^\theta(f,g)(x) \, \rd x = \int_{\mathbb R^d} g(z) \int_{|y|\le 2^j} f(z-y) \, \rd y \, \rd z \le C \|g\|_1  \|f\|_p 2^{dj\left(1-\frac{1}{p}\right)}
\]
and thus by the Cauchy-Schwarz inequality we have
\begin{equation}\label{L1}
\big\| I_j^\theta(f,g)\chi_E \big\|_{\frac12} \le  |E| \big\| I_j^\theta(f,g) \big\|_{1}\le
C |E| \|f\|_p  \|g\|_1 2^{dj\left(1-\frac{1}{p}\right)}.
\end{equation}
On the other hand, using \eqref{aux21} and H\"older's inequality we obtain
\begin{equation}\label{L2}
\big\| I_j^\theta(f,g)\chi_E \big\|_{\frac12} \le |E|^{1-\frac1p}\big\| I_j^\theta(f,g)  \big\|_{\frac{p}{p+1}}
\le C |E|^{1-\frac1p} 2^{dj} \|f\|_p \|g\|_1.
\end{equation}
Combining \eqref{L1} and \eqref{L2} gives \eqref{L0}. An analogous argument based on \eqref{aux20} yields \eqref{L00}. 
\par
We now proceed to the proof of \eqref{L000}. Using H\"{o}lder's inequality we have 
\[\big\| I_j^\theta(f,g)\chi_E \big\|_{\frac12} \leq |E|^{2 - \frac{\alpha}{d}} \big\| I_j^\theta(f,g) \big\|_{\frac{d}{\alpha}} \, . \]
By Jensen's inequality, Fubini's theorem, and the change of variables $z = x + \theta y$ we deduce 
\begin{align*}
\big\| I_j^\theta(f,g) \big\|_{\frac{d}{\alpha}}^{\frac{d}{\alpha}} & = \int_{\R^d} \left( \int_{|y|\leq 2^j} f(x + (\theta - 1)y) \, g(x + \theta y) \, \rd y \right)^{\frac{d}{\alpha}} \rd x \\
& \leq  C 2^{dj \left(\frac{d}{\alpha} - 1 \right)} \|f \|_{\infty}^{\frac{d}{\alpha}} \int_{|y| \leq 2^j} \int_{\R^d} g(x + \theta y )^{\frac{d}{\alpha}} \, \rd x \, \rd y \\
& = C 2^{dj \frac{d}{\alpha}} \|f \|_{\infty}^{\frac{d}{\alpha}} \|g \|_{\frac{d}{\alpha}}^{\frac{d}{\alpha}}
\end{align*}
and hence
\begin{align} \label{L3}
\big\| I_j^\theta(f,g)\chi_E \big\|_{\frac12} \leq C 2^{dj} |E|^{2 - \frac{\alpha}{d}} \|f \|_{\infty} \|g\|_{\frac{d}{\alpha}} \, .
\end{align}
Moreover, by H\"{o}lder's inequality and \eqref{aux20} with $p = d/\alpha$,
\begin{align} \label{L4}
\big\|I_j^\theta(f,g) \chi_E \big\|_{\frac{1}{2}} \leq |E|^{1-\frac{\alpha}{d}} \big\|I_j^\theta(f,g)  \big\|_{\frac{d}{d+\alpha}} \leq c 2^{dj} |E|^{1-\frac{\alpha}{d}} \|f\|_1 \|g\|_{\frac{d}{\alpha}} \, .
\end{align}
Additionally, by the Cauchy-Schwarz inequality and Fubini's theorem,
\begin{align*}
\big\|I_j^\theta(f,g) \chi_E \big\|_{\frac{1}{2}} & \leq |E| \big\|I_j^\theta(f,g) \chi_E \big\|_{1} \\
& \leq |E| \int_{B(2^j)} \int_E f(x + (\theta - 1)y) g(x + \theta y) \, \rd x \, \rd y \, .
\end{align*}
Using the change of variables $z = x+(\theta - 1)y$ we see that the previous expression equals
\[|E| \int_{B(2^j)} \int_{E + (\theta - 1)y} f(z) g(z + y) \, \rd z \, \rd y  \]
which, in turn, is bounded by
\[|E| \int_{\R^d} f(z) \int_{\frac{z - E}{\theta - 1}} g(z + y) \, \rd y \, \rd z \, .  \]
Applying H\"{o}lder's inequality to the inner integral above yields that the whole expression is bounded by 
\[|E| \|f\|_1 \|g\|_{\frac{d}{\alpha}} \left| \frac{E}{\theta - 1} \right|^{1- \frac{\alpha}{d}} = |E| \|f \|_1 \|g\|_{\frac{d}{\alpha}} |E|^{1- \frac{\alpha}{d}} |\theta - 1|^{\alpha-d}   \]
and so we conclude
\begin{align} \label{L5}
\big\|I_j^\theta(f,g) \chi_E \big\|_{\frac{1}{2}} \leq (1-\theta)^{\alpha - d} |E|^{2 - \frac{d}{\alpha}} \|f \|_1 \|g \|_{\frac{d}{\alpha}} \, .
\end{align}
Now, for $1 \leq p \leq \infty$, interpolating \eqref{L3} and \eqref{L4} results in 
\begin{align} \label{L6}
\big\|I_j^\theta(f,g) \chi_E \big\|_{\frac{1}{2}} \leq C 2^{dj} |E|^{2-\frac{\alpha}{d}-\frac{1}{p}} \|f \|_p \|g \|_{\frac{d}{\alpha}},
\end{align}
and similarly, interpolation between \eqref{L3} and \eqref{L5} yields
\begin{equation} \label{L7}
\big\|I_j^\theta(f,g) \chi_E \big\|_{\frac{1}{2}} \leq C (1-\theta)^{\frac{\alpha - d}{p}} 2^{dj \left(1 - \frac{1}{p} \right) } |E|^{2 - \frac{\alpha}{d}} \|f \|_p \| g \|_{\frac{d}{\alpha}} \, .
\end{equation}
Estimate \eqref{L000} follows at once from \eqref{L6} and \eqref{L7}. Analogously, we deduce \eqref{L0000}.
\end{proof}

We conclude this section with two useful auxiliary estimates. 
\begin{lemma}\label{SUM}
Let $d \in \mathbb{N}$, $0 < \alpha < d$, $1 \leq p \leq d/\alpha$, and $R,S>0$. There exists a positive constant $C = C(\alpha,d,p)$ such that:
\begin{enumerate}[(i)]
\item If $p < d/\alpha$, then 
\begin{equation} \label{A1}
\left(\sum_{j \in \mathbb{Z}} 2^{(\alpha - \frac dp)\frac{j}{2}} \min\{2^{dj}, R \}^{\frac{1}{2p}} \right)^2 \leq C\,  R^{\frac{\alpha}{d}} \, .
\end{equation}
\item If $1 < p$, then 
\begin{equation} \label{A2}
\sum_{j \in \mathbb{Z}} 2^{\alpha(1-p)j} \min\{2^{\alpha p j}S, R \} \leq C\, S \left( \frac{R}{S} \right)^{\frac{1}{p}} \, .
\end{equation}
\end{enumerate}
\end{lemma}
The proof of the inequalities in Lemma~\ref{SUM} follows by splitting the indices $j$ into the ranges where $2^j < R^{1/d}$ and the complementary range in \eqref{A1}, and where $2^j < (R/S)^{1/(\alpha p)}$ and its complement in \eqref{A2}. In both cases, the corresponding series converge and yield the stated bounds; see also \cite[Lemma 4.7]{agt}, which corresponds to the cases $p = 1$ in \eqref{A1} and $p = d/\alpha$ in \eqref{A2}. We note that the series on the left-hand side of \eqref{A1} diverges for $p \geq d/\alpha$, and the series on the left-hand side of \eqref{A2} diverges for $p \leq 1$.
\section{Proofs of Theorems \ref{thm_weak_bounds} and \ref{thm_strong_bounds}} \label{section_proofs}
In this section, we provide the proofs of Theorem \ref{thm_weak_bounds} and Theorem \ref{thm_strong_bounds}. Recall that if $h$ is a function in $L^{r,\infty}(\R^d)$ for some $0 < r < \infty$, then
\begin{equation} \label{weaknorm}
\|h \|_{r, \infty} \leq \sup_{0 < |E| < \infty} |E|^{-\frac{1}{s} + \frac{1}{r}} \|h \chi_E \|_s
\end{equation}
where $0 < s < r$ and the supremum is taken over measurable sets $E \subseteq \mathbb{R}^d$ of finite measure~\cite{grafakos2014classical}.\par 

Expressing $\R^d$ as the union of annuli,
\[\R^d = \bigcup_{j \in \Z} B(2^j) \setminus B(2^{j-1}) \]
we deduce the following pointwise estimate for $I_\alpha^\theta$ in terms of the auxiliary operator $I_j^\theta$ \eqref{Ithetaj}:
\begin{equation} \label{IboundedIj}
I_\alpha^\theta(f,g)(x) \leq c \sum_{j \in \Z} 2^{(\alpha - d)j} \, I_j^\theta(f,g)(x) \, .
\end{equation}

\medskip
\begin{proof}[Proof of Theorem \ref{thm_weak_bounds} $(i),(ii),(iii),(iv)$]
~
\par 
We start by proving part $(i)$. Let us fix $1\le p<d/\alpha$ and $r$ such that $1/r +\alpha/d = 1/p +1$. Using \eqref{weaknorm} with $s=1/2$, the pointwise estimate \eqref{IboundedIj}, and \eqref{L0} yields
\begin{align*}
\big\|I_\alpha^\theta(f,g) \big\|_{r, \infty} & \leq \sup_{0 < |E| < \infty} |E|^{-2 + \frac{1}{r}} \big\|I_\alpha^\theta(f,g) \chi_E \big\|_{\frac{1}{2}} \\
& \leq c \sup_{0 < |E| < \infty} |E|^{\frac{1}{p}-1-\frac{\alpha}{d}} \left( \sum_{j \in \Z} 2^{\frac{(\alpha - d)j}{2}} \big\| I_j^\theta(f,g) \chi_E  \big\|_\frac{1}{2}^\frac{1}{2} \right)^2 \\
& \leq C \sup_{0 < |E| < \infty} |E|^{-\frac{\alpha}{d}} \|f\|_p \|g\|_1 \left(\sum_{j \in \Z} 2^{ \left(\alpha-\frac{d}{p} \right)\frac{j}{2}}
\min \{2^{dj} , |E|\}^{\frac1{2p}}    \right)^2 \\
& \leq C \, \|f \|_p \|g \|_1 ,
\end{align*}
where we have used \eqref{A1} with $R = |E|$. The proof of part $(ii)$ is analogous and is based on \eqref{L00}. \par 
We now focus on part $(iii)$. Assume that $0 \leq \theta \leq 1- \delta$ and $1 \leq p < d/\alpha$. Then, by \eqref{weaknorm} with $s=1/2$, \eqref{IboundedIj}, and \eqref{L000},
\begin{align*}
\big\|I_\alpha^\theta(f,g) \big\|_{p, \infty} & \leq \sup_{0 < |E| < \infty} |E|^{-2 + \frac{1}{p}} \big\|I_\alpha^\theta(f,g) \chi_E \big\|_{\frac{1}{2}} \\
& \leq c \sup_{0 < |E| < \infty} |E|^{-2+\frac{1}{p}} \left( \sum_{j \in \Z} 2^{\frac{(\alpha - d)j}{2}} \big\| I_j^\theta(f,g) \chi_E  \big\|_\frac{1}{2}^\frac{1}{2} \right)^2 \\
& \leq C \sup_{0 < |E| < \infty} |E|^{-\frac{\alpha}{d}} \|f\|_p \|g\|_1 \left(\sum_{j \in \Z} 2^{ \left(\alpha-\frac{d}{p} \right)\frac{j}{2}}
\min \left\{2^{dj} , \frac{|E|}{(1-\theta)^{d-\alpha}} \right\}^{\frac1{2p}}    \right)^2 \\
& \leq C \delta^{(\alpha - d) \frac{\alpha}{d}} \, \|f \|_p \|g \|_{\frac{d}{\alpha}},
\end{align*}
where we   used \eqref{A1} with $R = (1-\theta)^{\alpha-d}|E|$. Part $(iv)$ is deduced similarly, based on \eqref{L0000}.
\end{proof}

\medskip

\begin{proof}[Proof of Theorem \ref{thm_strong_bounds}]
~
\par
It suffices to obtain restricted weak-type estimates for $I_\alpha^\theta$, with bounds depending on 
positive powers of $\theta^{-1}$ and $(1-\theta)^{-1}$, at the following points:
\[(1/p,1/q) \in \{(1,1),(1,0),(0,1), (\alpha/d,0), (0, \alpha/d)\}  \]
and corresponding $r$ given by \eqref{r_exponent}. The result then follows by Marcinkiewicz interpolation, Theorem~\ref{Minterpolation} in the appendix. \par 
The required estimate at the point $(1,1)$ is an immediate consequence of \eqref{ceiling} with $p = 1$ (or equivalently, \eqref{outer_wall} with $q=1$). \par 
Regarding the points $(1,0),(0,1)$, we prove that 
\begin{empheq}[left={\displaystyle \big\|I_\alpha^\theta(f,g) \big\|_{\frac{d}{d - \alpha}, \infty} \leq \, \empheqlbrace}]{align} 
    & c \, (1-\theta)^{-\alpha} \, \|f \|_1 \, \|g \|_\infty                 \, , \label{weak4} \\
    & c \, \theta^{-\alpha} \, \|f \|_\infty \, \|g \|_1               \, . \label{weak5}
  \end{empheq}
\noindent
To establish \eqref{weak4} we note that if $g \in L^\infty(\R^d)$, then 
\[
I_\alpha^\theta(f,g)(x)   \leq \|g \|_\infty \int_{\R^d} f(x + (\theta - 1)y) |y|^{\alpha - d} \, \rd y  
  = |\theta - 1|^{-\alpha} \|g \|_\infty \int_{\R^d} f(z) |z-x|^{\alpha - d} \, \rd z,
\]
where in the second step we used the change of variables $z = x + (\theta - 1)y$. Hence \eqref{weak4} follows from classical fractional integration. Estimate \eqref{weak5}  is deduced similarly. \par 
At last, we deduce restricted weak-type bounds for $I_\alpha^\theta$ at $(\alpha/d,0)$ and $(0,\alpha/d)$:
\begin{empheq}[left={\displaystyle \big\|I_\alpha^\theta(\chi_A,\chi_B) \big\|_\infty \leq \, \empheqlbrace}]{align} 
    & c \, (1-\theta)^{-\alpha} \, |A|^{\frac{\alpha}{d}}                  \, , \label{resweak1} \\
    & c \, \theta^{-\alpha} \, |B|^{\frac{\alpha}{d}}  \, , \label{resweak2} 
  \end{empheq}
where $A$ and $B$ are measurable subsets of $\R^d$ of finite measure. As the estimates are analogous, we only provide a proof of the first. Note that 
\[
I_\alpha^\theta(\chi_A, \chi_B)(x)   \leq \int_{\R^d} \chi_A(x + (\theta - 1)y) |y|^{\alpha - d} \, \rd y  
 = \int_{\frac{-x + A}{\theta - 1}} |y|^{\alpha - d} \, \rd y \, .
\]
Since the function $y \mapsto |y|^{\alpha - d}$ blows up at the origin (and it is radially symmetric), we can estimate the integral above by 
\[\int_{B(R)} |y|^{\alpha - d} \, \rd y \quad \text{with} \quad R = \frac{c|A|^{\frac1d}}{(1 - \theta)} \]
so that $\left|\frac{-x + A}{\theta - 1}\right| \leq |B(R)|$. Thus,
\begin{align*}
I_\alpha^\theta(\chi_A, \chi_B)(x) & \leq \int_0^R \int_{\partial B(\rho)} |y|^{\alpha - d} \, \rd S(y) \, \rd \rho = \frac{c}{(1-\theta)^\alpha} \, |A|^{\frac{\alpha}{d}} \, .
\end{align*}
\end{proof}
\medskip
\begin{proof}[Proof of Theorem \ref{thm_weak_bounds} $(v)$]
~
\par
Since $(\alpha/d, \alpha/d)$ belongs to the interior of the convex hull of \[\{(1,0),(0,1),(\alpha/d,0), (0,\alpha/d), (1,1)\}\] we have, by Theorem \ref{thm_strong_bounds},
\begin{equation*} \label{weak6}
\big\| I_\alpha^\theta(f,g) \big\|_{\frac{d}{\alpha}, \infty} \leq \big\| I_\alpha^\theta(f,g) \big\|_{\frac{d}{\alpha}} \leq \frac{c}{\theta^{\kappa_1}(1-\theta)^{\kappa_2}} \, \|f \|_{\frac{d}{\alpha}} \, \| g \|_{\frac{d}{\alpha}}
\end{equation*}
for some positive constants $\kappa_1, \kappa_2$ depending only on $\alpha$ and $d$.
\end{proof}

\section{Proof of Theorem \ref{thm_lorentz}} \label{section_proof_lorentz}

In this section we provide a proof of Theorem \ref{thm_lorentz}. First, we need a lemma.
\begin{lemma}
Let $1 < p \leq d/\alpha$ and $A$,$B$ and $E$ be measurable subsets of $\R^d$ of finite measure. Then, the auxiliary operator $I_j^\theta$ defined in \eqref{Ithetaj} satisfies
\begin{empheq}[left={\displaystyle \int_E I_j^\theta(\chi_A, \chi_B) \, \rd x \leq \, \empheqlbrace}]{align} 
    & C \, 2^{dj\left(1 - \frac{\alpha p}{d} \right)} \min \left\{ 2^{\alpha p j} |E|, |A| |B|^{\frac{\alpha p}{d}} \right\}                  \, , \label{aux1_lor} \\
    & C \, 2^{dj\left(1 - \frac{\alpha p}{d} \right)} \min \left\{ 2^{\alpha p j} |E|, |A|^{\frac{\alpha p}{d}} |B| \right\}  \, , \label{aux2_lor} 
  \end{empheq}
for some $C = C(\alpha,d,p) > 0$.
\end{lemma}
\begin{proof}
We first note that if $p = d/\alpha$ then \eqref{aux1_lor} and \eqref{aux2_lor} coincide; this case is handled in \cite[Lemma 4.4]{agt}. Therefore we assume for the rest of the proof that $p < d/\alpha$. Due to symmetry we only prove \eqref{aux1_lor}. Using Fubini's theorem, the change of variables $z = x + (\theta - 1)y$, and H\"{o}lder's inequality with $d/(\alpha p) > 1$ we deduce
\begin{align*}
\int_E I_j^\theta(\chi_A, \chi_B) \, \rd x & \leq  \int_{\R^d} \chi_A(z) \int_{\R^d} \chi_B(z+y) \chi_{B(2^j)}(y) \, \rd y \, \rd z \\
& \leq \int_{\R^d} \chi_A(z) \left(\int_{\R^d} \chi_B(z+y) \, \rd y \right)^{\frac{\alpha p}{d}} \left(\int_{\R^d} \chi_{B(2^j)}(y)\, \rd y \right)^{1-\frac{\alpha p}{d}} \\
& = C \, |A| \, |B|^{\frac{\alpha p}{d}} \, 2^{dj\left(1-\frac{\alpha p}{d} \right)} \, .
\end{align*}
Estimate \eqref{aux1_lor} follows from the bound above together with the fact that 
\begin{align*}
\int_E I_j^\theta(\chi_A, \chi_B) \, \rd x \leq c 2^{dj} |E| = c \, 2^{dj\left(1-\frac{\alpha p}{d} \right)} \, 2^{\alpha p j} \, |E| \, .
\end{align*}
\end{proof}

\begin{proof}[Proof of Theorem \ref{thm_lorentz}]
~
\par

In view of symmetry we only prove the estimate in \eqref{floor_uniform}. We first assume that $1 < p \leq d/\alpha$. Note that in this case the weak space $L^{p,\infty}(\R^d)$ is a Banach space. Therefore, when $1 < p \leq d/\alpha$ it suffices to show that \eqref{floor_uniform} holds when $f$ and $g$ are characteristic functions of sets of finite measure; see Theorem~\ref{Linterpolation} in the appendix. Let $A$ and $B$ be such sets. Using~\eqref{weaknorm} with $s=1$, \eqref{IboundedIj} and \eqref{aux1_lor} we deduce
\begin{align*}
\big\|I_\alpha(\chi_A, \chi_B)  \big\|_{p,\infty} & \leq c \sup_{0 < |E| < \infty} |E|^{-1 + \frac{1}{p}} \sum_{j \in \Z} 2^{(\alpha - d)j} \int_E I_j^\theta(f,g) \, \rd x \\
& \leq C \sup_{0 < |E| < \infty} |E|^{-1 + \frac{1}{p}} \sum_{j \in \Z} 2^{\alpha(1-p)j} \min \left\{ 2^{\alpha p j} |E|, |A| |B|^{\frac{\alpha p}{d}} \right\} \\
& \leq C \, |A|^{\frac{1}{p}} \, |B|^{\frac{\alpha}{d}}, 
\end{align*}
where in the last step we have used \eqref{A2} with $S = |E|$ and $R = |A||B|^{\frac{\alpha p}{d}}$. Analogously, one deduces the estimate in \eqref{inner_wall_uniform} with $1 < q \leq d/\alpha$, based on \eqref{aux2_lor}. \par 
We now turn our attention to the case $p = 1$ in \eqref{floor_uniform} (the case $q = 1$ in \eqref{inner_wall_uniform} is obtained similarly). By appropriate changes of variables we have 
\begin{align*}
\big\|I_\alpha^\theta(f,g) \big\|_{1,\infty} & \leq \big\|I_\alpha^\theta(f,g) \big\|_1 \\
&=\int_{\mathbb R^d} \int_{\mathbb R^d} f(x+(\theta-1)y) \, g(x+\theta y) |y|^{\alpha-d} \,\rd  x \,\rd  y\\ 
& = \int_{\mathbb R^d}f(z) \int_{\mathbb R^d} g(z+y) \, |y|^{\alpha-d} \,\rd  y \,\rd  z\\   
& \leq \|f \|_1 \, \| I_\alpha(g) \|_\infty \, .
\end{align*}
The desired result follows by duality: the linear fractional integral $I_\alpha$ maps $L^{\frac{d}{d-\alpha},\infty}(\R^d)$ to $L^1(\R^d)$, therefore it maps $L^{\frac{d}{\alpha},1}(\R^d)$ to $L^\infty(\R^d)$.
\end{proof}

\section{Sharpness of Theorem \ref{thm_weak_bounds}} \label{section_sharp} 
In this section, we investigate the sharpness of all the estimates in Theorem \ref{thm_weak_bounds}. To do so, we make use of two auxiliary functions. Let $\Phi \in C_c^\infty(\R^d)$ be a nonnegative radial function with $\operatorname{supp}(\Phi) \subseteq B(1/8)$. Additionally, we consider a function $h$ on $\R^d$ defined by 
\begin{equation} \label{function_h}
h(x) = |x|^{-\alpha} \left(\log  \frac{1}{|x|}   \right)^{-\kappa} \chi_{|x| \leq 1/e}(x), 
\end{equation}
where $\kappa = (d+\alpha)/(2d)$. Note that $\alpha/d < \kappa < 1$ since $\alpha < d$. One can readily see that $h$ belongs to $L^{\frac{d}{\alpha}}(\R^d)$. Moreover, we have the following lower bound for the linear fractional integral of $h$:
\begin{align*}
I_\alpha (h)(x) & = \int_{|y| \leq 1/e} |y|^{-\alpha} \left(\log  \frac{1}{|y|}   \right)^{-\kappa} |x - y|^{\alpha - d} \, \rd y \\
 & \geq \int_{|x| \leq |y| \leq 1/e} |y|^{-\alpha} \left(\log  \frac{1}{|y|}   \right)^{-\kappa} |x - y|^{\alpha - d} \, \rd y \, \chi_{|x| \leq 1/8}(x) \\
 & \geq c \int_{|x| \leq |y| \leq 1/e} |y|^{-d}  \, \rd y \left(\log  \frac{1}{|x|}   \right)^{-\kappa} \chi_{|x| \leq 1/8}(x) \\
 & = c \left( \log \frac{1}{|x|} - 1 \right) \left(\log  \frac{1}{|x|}   \right)^{-\kappa} \chi_{|x| \leq 1/8}(x) \, .
\end{align*}
Observe that for $|x| \leq 1/8 < e^{-2}$ we have
\[\log \frac{1}{|x|} - 1 = \log \frac{1}{|x|} \left( 1 + \frac{1}{\log |x|}  \right) \geq \frac{1}{2} \log \frac{1}{|x|}  \]
and so, it follows from the above estimate that
\begin{equation} \label{h_lower_bound}
I_\alpha (h)(x) \geq c \left(\log  \frac{1}{|x|}   \right)^{1-\kappa} \chi_{|x| \leq 1/8}(x) 
\end{equation}
for $x \in \R^d$.  

\vspace{2mm}
\noindent
{\bf Case I.} We begin with the fact that we cannot take $p=d/\alpha$  in \eqref{ceiling}. Suppose, towards a contradiction, that this were the case. Then, by letting $\theta \to 0$ and using Fatou's lemma (which also holds for weak $L^p$ spaces), we would obtain
\begin{equation}\label{sharp1}
\| I_\alpha (h) \varphi_t  \|_{1,\infty} \le C \|\varphi_t\|_1 
\end{equation}
for all $t > 0$, where $\varphi_t(x) = t^{-d} \Phi(x/t)$. The right-hand side of \eqref{sharp1} is constant, equal to $C \|\Phi \|_1$; we prove that the left-hand side tends to infinity as $t \to 0$, leading to a contradiction. Assuming that $0 < t < 1$ and using \eqref{h_lower_bound} we deduce
\begin{align*}
I_\alpha(h)(x) \varphi_t(x) & \geq c \left( \log \frac{1}{|x|} \right)^{1-\kappa} t^{-d} \Phi(x/t) \\
& = c \left(\log \frac{1}{t} + \log \frac{t}{|x|}  \right)^{1-\kappa} t^{-d} \Phi(x/t) \\
& \geq c \left( \log \frac{1}{t} \right)^{\frac{1-\kappa}{2}} \left( \log \frac{t}{|x|} \right)^{\frac{1-\kappa}{2}}  t^{-d} \Phi(x/t)
\end{align*}
and hence
\begin{equation*}
\big\|I_\alpha(h) \varphi_t \big\|_{1,\infty} \geq c \left( \log \frac{1}{t} \right)^{\frac{1-\kappa}{2}} \big\|(\log | \cdot |^{-1})^{\frac{1-\kappa}{2}} \Phi \big\|_{1,\infty}
\end{equation*}
where we have used that $L^{1,\infty}(\R^d)$ has the same dilation structure as $L^1(\R^d)$. It is straightforward to verify that $\big\|(\log | \cdot |^{-1})^{\frac{1-\kappa}{2}} \Phi \big\|_{1,\infty}$ is a finite constant. Therefore, letting $t \to 0$ in the last inequality above yields the desired conclusion.

\vspace{2mm}
\noindent
{\bf Case II.} In a similar fashion (but letting $\theta \to 1$ instead) we obtain the sharpness of \eqref{outer_wall}, that is, one cannot take $q=d/\alpha$ in that inequality.

\vspace{2mm}
\noindent
{\bf Case III.} Next, we obtain the sharpness of \eqref{floor} in terms of $\delta$. If \eqref{floor} were valid for $\delta =0$, then, by the same argument as in Case I, we would have
\begin{equation}\label{sharp2}
\| \psi_t I_\alpha (h)  \|_{p,\infty} \le C \|\psi_t\|_p 
\end{equation}
for all $t > 0$, where $\psi_t(x) = t^{-d/p} \Phi(x/t)$. The right-hand side of \eqref{sharp2} equals $C\|\Phi\|_p$ and is therefore finite. We estimate the left-hand side using \eqref{h_lower_bound} (and assuming $0 < t < 1$):
\begin{equation*}
\big\|\psi_t I_\alpha(h) \big\|_{p,\infty} \geq c \left( \log \frac{1}{t} \right)^{\frac{1-\kappa}{2}} \big\|(\log | \cdot |^{-1})^{\frac{(1-\kappa)p}{2}} \Phi^p \big\|_{1,\infty}^{1/p}
\end{equation*}
which blows-up as $t \to 0$, a contradiction.

\vspace{2mm}
\noindent
{\bf Case IV.} Analogously to Case III, we obtain the sharpness of \eqref{inner_wall}, that is, one cannot take $\delta = 0$ in that inequality.

\vspace{2mm}
\noindent
{\bf Case V.} The sharpness of \eqref{bad_corner} in terms of $\delta$; that is, the fact that we cannot take $\delta =0$ in that estimate, is achieved as in Case III with $p = d/\alpha$. \par

\appendix
\section{Interpolation theorems} \label{appendix1}
In this section we collect the interpolation results used in this work for the reader's convenience. First, we recall that a bilinear operator $T$ acting on measurable functions is said to be of \textit{restricted weak-type} $(p,q,r)$ (with constant $c>0$) if 
\[ \| T(\chi_A, \chi_B) \|_{r,\infty} \leq c \, |A|^{\frac{1}{p}} \, |B|^{\frac{1}{q}} \]
for all measurable sets $A$ and $B$ of finite measure. The integrability exponents $p,q,r$ belong to $(0,\infty]$.

The first interpolation result we state is the Marcinkiewicz interpolation theorem. It provides strong-type bounds from a finite set of restricted weak-type estimates and is used in the proofs of Theorem~\ref{thm_uniform_strong_bounds} in~\cite{agt} and Theorem~\ref{thm_strong_bounds}. General multilinear versions of this interpolation theorem are proved in~\cite{grafakos2012multilinear} and~\cite{grafakos2014modern}. 

\begin{theorem} \label{Minterpolation}
Let $0<p_i, q_i, r_i \leq \infty$ for $i = 1,2,3$. Suppose that the points 
\[
\Big( \frac{1}{p_1}, \frac{1}{q_1}\Big), \quad
\Big( \frac{1}{p_2}, \frac{1}{q_2}\Big), \quad
\Big( \frac{1}{p_3}, \frac{1}{q_3}\Big),
\]
do not lie on the same line in $\mathbb R^2$. 
For $0<\theta_1, \theta_2,\theta_3  <1$ satisfying $\theta_1+ \theta_2+\theta_3 =1$ consider the points $0<p,q,r\leq \infty$ such that
\[
\Big( \frac{1}{p }, \frac{1}{q },\frac{1}{r }\Big)=
\theta_1\Big( \frac{1}{p_1}, \frac{1}{q_1},\frac{1}{r_1}\Big)+
\theta_2\Big( \frac{1}{p_2}, \frac{1}{q_2},\frac{1}{r_2}\Big)+
\theta_3\Big( \frac{1}{p_3}, \frac{1}{q_3},\frac{1}{r_3}\Big)
\]
and 
\[
\frac{1}{r} \le \frac{1}{r_1}+\frac{1}{r_2}+\frac{1}{r_3} \, .
\]
Let $T$ be a bilinear operator that is of restricted weak-type $(p_i,q_i,r_i)$ (with constant $c_i>0$) for $i=1,2,3 $. Then there is a constant $C >0$ 
depending only on $p_i$, $q_i$, $r_i$, and $\theta_i$ ($i=1,2,3 $) such that 
\[
\| T(f, g)\|_{L^r} \le C \, c_1^{\theta_1} c_2^{\theta_2} c_3^{\theta_3} 
\| f\|_{L^p} \| g\|_{L^q} 
\]
for all functions $f\in L^p(\R^d)$ and $g \in L^q(\R^d)$.
 \end{theorem}

The next result concerns interpolation of bilinear operators, in the sense that if an operator satisfies strong-type bounds at two points, then it satisfies strong-type bounds at all points in between. This result follows from general multilinear complex interpolation theorems (see~\cite[Theorem 7.2.9, Corollary 7.2.11]{grafakos2014modern} and~\cite[Theorem 3.2]{grafakos2022interpolation}), and it is used in the proofs of Lemmas~\ref{lem_00K} and~\ref{OOKK}.
\begin{theorem} \label{Cinterpolation}
Let $0 < p_i,q_i, r_i \leq \infty$ ($i=1,2$) and let $T$ be a bilinear operator acting of measurable functions such that 
\[\|T(f,g)\|_{r_i} \leq c_i \, \|f \|_{p_i} \, \|g\|_{q_i}\]
for all functions $f\in L^{p_i}(\R^d)$ and $g \in L^{q_i}(\R^d)$ and some positive constants $c_i$ ($i = 1,2$). Then for $0 < \varphi < 1$ and $p,q,r$ such that 
\[\frac{1}{p} = \frac{1- \varphi}{p_1} + \frac{\varphi}{p_2}, \qquad \frac{1}{q} = \frac{1- \varphi}{q_1} + \frac{\varphi}{q_2}, \qquad \frac{1}{r} = \frac{1- \varphi}{r_1} + \frac{\varphi}{r_2}\]
we have
\[\|T(f,g)\|_{r} \leq  c_1^{1 - \varphi} \, c_2^{\varphi} \, \|f \|_{p} \, \|g\|_{q}\]
for all $f\in L^{p}(\R^d)$ and $g \in L^q(\R^d)$.
\end{theorem}

At last, we present a result that yields sufficient conditions for a bilinear operator to admit an extension to the product of Lorentz spaces $L^{p,1}(\R^d) \times L^{q,1}(\R^d)$, $0 < p,q < \infty$, being the (Euclidean) bilinear counterpart of the result in \cite[Exercise~1.4.7]{grafakos2014classical}. We use this result in the proof of Theorem~\ref{thm_lorentz}.

\begin{theorem} \label{Linterpolation}
Let $Z$ be a Banach space of real-valued measurable functions on $\R^d$, and let $T$ be a bilinear operator defined on the space of finitely simple functions on $\R^d$ and taking values in $Z$. Suppose that for $0 < p,q < \infty$ and some constant $C>0$ the following restricted weak-type estimate
\[\|T(\chi_A, \chi_B) \|_Z \leq C \, |A|^{\frac{1}{p}} \, |B|^{\frac{1}{q}}\]
holds for all measurable sets $A,B$ of finite measure. Then, $T$ has a bounded extension from $L^{p,1}(\R^d)\times L^{q,1}(\R^d)$ to $Z$.
\end{theorem}

We note that the proof of the theorem above is straightforward and follows exactly the idea as in the linear case; the standard proof is omitted. 

\section*{Acknowledgments}
The authors gratefully acknowledge Grigorios Kounadis for his assistance with the \LaTeX{} code used to illustrate the regions of boundedness. This research was partially funded by the Austrian Science Fund (FWF), project number 10.55776/F65.

\end{document}